\documentclass[10pt]{amsart}

\usepackage{hyperref}
\hypersetup{colorlinks=true, urlcolor=blue, citecolor=blue, linkcolor=blue}

\usepackage{amsmath,amsthm, amssymb}

\usepackage{graphicx,mathrsfs,tikz,fancyhdr}

\usepackage{thmtools}
\usepackage{nameref}
\usepackage[backend=bibtex]{biblatex}

\bibliography{Masseybib.bib}

\newcommand{\U}{{\mathcal U}}
\newcommand{\0}{{\mathbf 0}}
\newcommand{\C}{{\mathbb C}}
\newcommand{\Z}{{\mathbb Z}}

\newcommand{\W}{{\mathcal W}}
\newcommand{\V}{{\mathcal V}}
\newcommand{\T}{{\mathcal T}}
\newcommand{\strat}{{\mathfrak S}}

\newcommand{\hyp}{{\mathbb H}}

\newcommand{\rank}{\mathop{\rm rank}\nolimits}

\newcommand{\Adot}{\mathbf A^\bullet}

\newcommand{\Cdot}{\mathbf C^\bullet}

\newcommand{\Ndot}{\mathbf  N^\bullet}

\newcommand{\vdual}{{\mathcal D}}

\newtheorem{defn0}{Definition}[section]
\newtheorem{prop0}[defn0]{Proposition}
\newtheorem{conj0}[defn0]{Conjecture}
\newtheorem{thm0}[defn0]{Theorem}
\newtheorem{lem0}[defn0]{Lemma}
\newtheorem{corollary0}[defn0]{Corollary}
\newtheorem{example0}[defn0]{Example}
\newtheorem{remark0}[defn0]{Remark}
\newtheorem{question0}[defn0]{Question}

\newenvironment{defn}{\begin{defn0}}{\end{defn0}}
\newenvironment{prop}{\begin{prop0}}{\end{prop0}}

\newenvironment{thm}{\begin{thm0}}{\end{thm0}}
\newenvironment{lem}{\begin{lem0}}{\end{lem0}}
\newenvironment{cor}{\begin{corollary0}}{\end{corollary0}}
\newenvironment{exm}{\begin{example0}\rm}{\end{example0}}
\newenvironment{rem}{\begin{remark0}\rm}{\end{remark0}}

\newcommand{\defref}[1]{Definition~\ref{#1}}
\newcommand{\propref}[1]{Proposition~\ref{#1}}
\newcommand{\thmref}[1]{Theorem~\ref{#1}}
\newcommand{\lemref}[1]{Lemma~\ref{#1}}
\newcommand{\corref}[1]{Corollary~\ref{#1}}
\newcommand{\exref}[1]{Example~\ref{#1}}
\newcommand{\secref}[1]{Section~\ref{#1}}

\newcommand{\mbf}[1]{{\mathbf #1}}

\title[Perverse Results on Parameterized Hypersurfaces]{Perverse Results on Milnor Fibers inside \\ Parameterized Hypersurfaces}

\author{Brian Hepler}
\address{Department of Mathematics\\
  Northeastern University\\
  Boston, Massachusetts 02115}
\email[B.~Hepler]{hepler.b@husky.neu.edu}
\author{David B. Massey}
\address{Department of Mathematics\\
  Northeastern University\\
  Boston, Massachusetts 02115}
\email[D.~Massey]{d.massey@neu.edu}
\begin{document}

\begin{abstract}
We discuss some results on parameterized hypersurfaces which follow quickly from results in the category of perverse sheaves.
\end{abstract}

\maketitle

\thispagestyle{fancy}

\lhead{}
\chead{}
\rhead{ }

\lfoot{}
\cfoot{}
\rfoot{}

\section{Introduction} Throughout this paper,  $\U$ will denote an open neighborhood of the origin in $\C^{n+1}$, $\W$ will denote an open subset of $\C^n$ (or an open subset of a finite number of disjoint copies of $\C^n$), $S$ will denote a finite set of $r$ points $\{p_1, \dots, p_r\}$ in $\W$, and $F:(\W, S)\rightarrow (\U, \0)$ will denote a finite, complex analytic map which is generically one-to-one such that $F^{-1}(\0)=S$.

We are interested in the germ of the image of $F$ at the origin. By shrinking $\W$ and $\U$ if necessary, we can, and do, assume  that $\W$ consists of $r$ disjoint, connected, open sets, $\W_1, \dots, \W_r$ and, for $1\leq i\leq r$, $F^{-1}(\0)\cap \W_i=\{p_i\}$. The case where $r=1$ is usually referred to as the {\it mono-germ} case, and the case where $r>1$ as the {\it multi-germ} case.

In our setting, the Finite Mapping Theorem \cite{grauertremmert} tells us that the image of $F$ is a complex analytic space of dimension $n$, i.e., is a hypersurface $X:=V(g)$ in $\U$, for some complex analytic $g:\U\rightarrow\C$. We will continue to use $F$ to denote the surjection $F:\W\rightarrow X$.

\medskip

Another way of thinking of $F:\W\rightarrow X$ is as a finite resolution of singularities. In particular, $F$ is a small resolution in the sense of Goresky-MacPherson and, consequently, the shifted constant sheaf $\Z_\W^\bullet[n]$ on $\W$ pushes forward by $F$ to the intersection cohomology complex ${\mathbf{I}}_{{}_X}^\bullet$ on $X$ (see \cite{inthom2}).

\medskip

The stalk cohomology of ${\mathbf{I}}_{{}_X}^\bullet$ is trivial to describe. For each $x\in X$, let $m(x)$ denote the number of points in the inverse image of $F$ (\textbf{without} multiplicity), i.e., $m(x):=\left|F^{-1}(x)\right|$.  Note that $m(\0)=r$. Then the stalk cohomology of ${\mathbf{I}}_{{}_X}^\bullet$ is given by, for all $x\in X$,
$$
H^k({\mathbf{I}}_{{}_X}^\bullet)_x\cong
\begin{cases}
\Z^{m(x)}, &\textnormal{ if } k=-n;\\
0, &\textnormal{ otherwise}.
\end{cases}
$$

\medskip

In this paper, we will use general properties and results from the derived category and the category of perverse sheaves to investigate the cohomology of Milnor fibers of complex analytic functions $h:X\rightarrow\C$. We outline these results below.

\vskip 0.3in

For $k\geq 1$,  let $X_k:=\{x\in X\ |\ m(x)=k\}$,
 and let 
$$
D:=\overline{\bigcup_{k\geq 2}X_k},
$$
which is the closure of the image of the double-point (or multiple-point) set with its reduced structure. Note that, since we are taking the closure, $D$ may contain points of $X_1$. Later, we shall show that $D$ is contained in the singular set $\Sigma X$ of $X$.

\medskip

Suppose that we have a complex analytic function $h: (X, \0)\rightarrow (\C, 0)$.

\bigskip

We are interested in results on the Milnor fiber, $M_{h,\0}$ of $h$ at $\0$. We remind the reader that, in this context in which the domain of $h$ is allowed to be singular,  a Milnor fibration still exists by the result of L\^e in \cite{relmono}, and the Milnor fiber at a point $x\in V(h)$, is given by
$$
M_{h,x} \ = \ B^\circ_\epsilon(x)\cap X\cap h^{-1}(a),
$$
where $B^\circ_\epsilon(x)$ is an open ball of radius $\epsilon$, centered at $x$, and $0<|a|\ll\epsilon\ll 1$ (and, technically, the intersection with $X$ is redundant, but we wish to emphasize that this Milnor fiber lives in $X$). We also care about the real link, $K_{{}_{X, x}}$, of $X$ at $x\in X$ \cite{milnorsing}, which is  given by
$$
K_{{}_{X, x}} \ := \ \partial B_\epsilon(x)\cap X =S_\epsilon(x)\cap X,
$$
where, again, $0<\epsilon\ll 1$.

\bigskip

We will need to consider the Milnor fiber of $h\circ F$ at each of the $p_i$  and the Milnor fiber of $h$ restricted to the $X_k$'s, which are equal to the intersections $X_k\cap M_{h, \0}$.

\bigskip

As $X$ itself may be singular, it is important for us to say what notion we will use for a ``critical point'' of $h$. We use the Milnor fiber to define:

\begin{defn}The topological/cohomological critical locus of $h$, is 
$$\Sigma_{{}_{\operatorname{top}}}h :=\{x\in V(h) \ | \ M_{h, x} \textnormal{ does not have the integral cohomology of a point}\}.$$
\end{defn}

\smallskip

\begin{rem}\label{rem:unstablelocus} Suppose, for instance, that $F$ is a stable unfolding of a finite map $f$, and that $h$ is the projection onto one of the unfolding parameters. Then a point  $x\in V(h)$ is a point in the image of $f$. If $f$ is stable at $x$, then $h$ is locally a topologically trivial fibration in a neighborhood of $x$; consequently, the Milnor fiber is contractible (as $V(h)$ is in a neighborhood  of a point), and $x\not\in \Sigma_{{}_{\operatorname{top}}}h$.

Thus, $\Sigma_{{}_{\operatorname{top}}}h$ is contained in the unstable locus of $f$.
\end{rem}

\bigskip

Now, $F$ induces a finite map $\widetilde F$ from each $M_{h\circ F, p_i}$ to $M_{h, \0}$, which can be stratified in the sense of Goresky and MacPherson \cite{inthom2} in such a way that the closure of each $X_k$ is a union of strata. From this, via a Riemann-Hurwitz-type argument, it is not difficult to show that the Euler characteristics are related by 
$$
\sum_{1\leq i\leq r}\chi(M_{h\circ F, p_i})=\sum_{k\geq 1}k\cdot\chi(X_k\cap M_{h, \0})= \chi(M_{h, \0})+\sum_{k\geq 2}(k-1)\cdot\chi(X_k\cap M_{h, \0}).
$$
Or, rearranging and writing $\widetilde \chi$ for the Euler characteristic of the reduced cohomology (in order to focus on vanishing cohomology), we obtain
\begin{equation}
\widetilde\chi(M_{h, \0}) =r-1+\sum_i \widetilde\chi(M_{h\circ F, p_i}) - \sum_{k\geq 2}(k-1)\cdot\chi(X_k\cap M_{h, \0}).\tag{$\star$}
\end{equation}

Equation ($\star$) is particularly interesting in the case where the reduced cohomology of $M_{h,\0}$ is concentrated in a single degree and the reduced cohomology of $M_{h\circ F,p_i}$ is zero.

\bigskip

In this paper, we show:
\begin{enumerate}

\item If $s:=\dim_\0\Sigma_{{}_{\operatorname{top}}}h$, then the reduced cohomology of $M_{h, \0}$ can be non-zero only in degrees $k$ where $n-1-s\leq k\leq n-1$, and is free Abelian in degree $n-1-s$. 

In particular, if $\0$ is an isolated point in $\Sigma_{{}_{\operatorname{top}}}h$, then $M_{h, \0}$ has the cohomology of a bouquet of $(n-1)$-spheres. 

\smallskip

\item As discussed above, there is a relationship between the reduced Euler characteristics of the Milnor fiber $M_{h,\0}$, the Milnor fibers $M_{h\circ F,p_i}$, and the Milnor fibers of the $X_k$'s, given by
$$
\widetilde\chi(M_{h, \0}) =r-1+\sum_i \widetilde\chi(M_{h\circ F, p_i}) - \sum_{k\geq 2}(k-1)\cdot\chi(X_k\cap M_{h, \0}).
$$

\smallskip

\item There is a perverse sheaf $\Ndot$, supported on $D$, with the properties that:

\smallskip

\begin{itemize} 
\item for all $x\in D$, the stalk cohomology of $\Ndot$ at $x$ is (possibly) non-zero in a single degree, degree $-n+1$, where it is isomorphic to $\Z^{m(x)-1}$;
\smallskip

\item With some special assumptions on $h$, there is a long exact sequence, relating the Milnor fiber of $h$, the Milnor fibers of $h\circ F$, and the hypercohomology of the Milnor fiber of $h$ restricted to $D$ with coefficients in $\Ndot[-n+1]$, given by

\smallskip

$\cdots\rightarrow \widetilde\hyp^{j-1}(D\cap M_{h,\0}; \Ndot[-n+1])\rightarrow \widetilde H^{j}(M_{h, \0};\Z)\rightarrow\hfill$

\smallskip

$\hfill \bigoplus_i\widetilde H^{j}(M_{h\circ F, p_i};\Z)\rightarrow \widetilde\hyp^{j}(D\cap M_{h,\0}; \Ndot[-n+1])\rightarrow\cdots \ ,$

\smallskip

where the reduced cohomology with coefficients in $\Ndot[-n+1]$ has the special meaning of reducing the rank by $r-1$ in degree zero and having no effect in other degrees.

\smallskip

\noindent This long exact sequence is compatible with the Milnor monodromy automorphisms in each degree.

\smallskip

\item In particular, if $S\cap \Sigma(h\circ F)=\emptyset$, then the reduced cohomology $\widetilde H^j(M_{h,\0}; \Z)$ is isomorphic to the reduced hypercohomology $\widetilde\hyp^{j-1}\left(D\cap M_{h, \0}; \Ndot[-n+1]\right)$, by an isomorphism which commutes with the respective Milnor monodromies.
\end{itemize}

\smallskip

\item Suppose that $\0$ is an isolated point in $\Sigma_{{}_{\operatorname{top}}}h$ and that $S\cap\Sigma(h\circ F)=\emptyset$. Then, 
$$\widetilde H^{n-1}(M_{h, \0}; \Z)\  \cong \ \Z^\omega \  \cong  \ \ \widetilde \hyp^{n-2}\left(D\cap M_{h, \0}; \Ndot[-n+1] \right),$$
where $\omega:= (-1)^{n-1}\left[(r-1) -\sum_{k\geq 2}(k-1)\chi(X_k\cap M_{h,\0})\right]$.

\smallskip

\item Suppose that $n=2$ and that $F$ is a one-parameter unfolding of a parameterization $f$ of a plane curve singularity with $r$ irreducible components at the origin. Let $t$ be the unfolding parameter and suppose that the only singularities of $M_{t, \0}$ are nodes, and that there are $\delta$ of them.  Recall that $X=V(g)$, and let $g_0:=g_{|_{V(t)}}$. Then, we recover the classical formula for the Milnor number of $g_0$, as given in Theorem 10.5 of \cite{milnorsing}:
$$
\mu_\0 \left ( g_0 \right )= 2 \delta -r + 1.
$$

\end{enumerate}

\bigskip

\section{A Standard Vanishing Result}

Before we state the only result of this section, we need to establish a convention for a degenerate case: the reduced cohomology of the  empty set.

\medskip

\noindent {\bf Convention}: We define the reduced cohomology of the empty set, $\widetilde H^k(\emptyset; \Z)$, to be zero in all degrees other than degree $-1$, and we define $\widetilde H^{-1}(\emptyset; \Z)=\Z$.

We do this so that the stalk cohomology at $p$ of the vanishing cycles of the constant sheaf along a complex analytic function $f:(E, p)\rightarrow (\C, 0)$ always yields the reduced cohomology of the Milnor fiber of $f$ at $p$, even in the case where $f$ is identically zero on $E$. This is true because, if $B$ is the intersection with $E$ of a small open ball around $p$ in some ambient affine space (after embedding), then
$$
H^k(\phi_f\Z_E^\bullet)_p\cong H^{k+1}(B, M_{f, p};\Z).
$$
One then looks at the long exact sequence of the pair $(B, M_{f, p})$, paying special attention to the case where $M_{f, p}=\emptyset$, i.e., the case where $f$ is identically zero.

\bigskip

The following result is, by now, a well-known consequence of the general theory of perverse sheaves and vanishing cycles. Nonetheless, we give a quick proof.

\begin{prop}\label{prop:leprop}  If $s:=\dim_\0\Sigma_{{}_{\operatorname{top}}}h$, then the reduced cohomology of $M_{h, \0}$ can be non-zero only in degrees $k$ where $n-1-s\leq k\leq n-1$, and is free Abelian in degree $n-1-s$. 

In particular, if $\0$ is an isolated point in $\Sigma_{{}_{\operatorname{top}}}h$, then $M_{h, \0}$ has the cohomology of a bouquet of $(n-1)$-spheres. 
\end{prop}
\begin{proof} 
By the result of L\^e in \cite{levan}, if $X$ is a purely $n$-dimensional local complete intersection, and $S$ is a $d$-dimensional stratum in a Whitney stratification of $X$, then the complex link of $S$ has the homotopy-type of a finite bouquet of $(n-1-d)$-spheres.

The cohomological implication is that the constant sheaves $\Z_X^\bullet[n]$ and $\left(\Z/p\Z\right)_X^\bullet[n]$, for $p$ prime, are perverse sheaves. Consequently, the shifted vanishing cycles 
$$\phi_h[-1]\Z_X^\bullet[n] \hskip 0.1in\textnormal{and}\hskip 0.1in\phi_h[-1]\left(\Z/p\Z\right)_X^\bullet[n]$$
 are also perverse, and have support contained in the closure $\overline{\Sigma_{{}_{\operatorname{top}}}h}$.
 
 Hence, these vanishing cycles have possibly non-zero stalk cohomology in degrees $k$ such that $-s\leq k\leq 0$. This means that the reduced cohomology of the Milnor fiber of $h$ at $\0$, with coefficients in $\Z$ or $\Z/p\Z$, is possibly non-zero in degrees $n-1-s$ and $n-1$. This proves the result, except for the free Abelian claim in degree $n-1-s$.
 
 However, that is the point of the $\Z/p\Z$ discussion. As $\widetilde H^{n-2-s}(M_{h,\0}; \Z/p\Z)=0$, for all $p$, the Universal Coefficient Theorem tells us that $\widetilde H^{n-1-s}(M_{h,\0}; \Z)$ has no torsion.
\end{proof}

\medskip

For a stable unfolding $F$ with an isolated instability and projection $h$ onto an unfolding parameter, the result above is a cohomological generalization of the result of Mond in \cite{mondimagemilnor}.

\bigskip

\begin{defn} If $\0$ is an isolated point in $\Sigma_{{}_{\operatorname{top}}}h$, then we define the {\bf Milnor number of $h$ at $\0$}, $\mu_\0(h)$, to be the rank of $\widetilde H^{n-1}(M_{h, \0}; \Z)$. 
\end{defn}

\bigskip

\section{The Push-forward of the Constant Sheaf}

General references for the derived category techniques in this section and the next are \cite{kashsch}, \cite{dimcasheaves}, and \cite{inthom2}. As we are always considering the derived category, we follow the usual practice of omitting the ``R''s in front of right derived functors.

\bigskip

We made the following definition in the introduction.

\begin{defn} Let ${\mathbf{I}}_{{}_X}^\bullet$ denote the (derived) push-forward of the constant sheaf $\Z_\W^\bullet[n]$; that is, ${\mathbf{I}}_{{}_X}^\bullet:=F_*\Z_\W^\bullet[n]$. 
\end{defn}

\smallskip

\noindent In the notation ${\mathbf{I}}_{{}_X}^\bullet$, we will justify subscripting by $X$, rather than by $F$, below.

\smallskip

\begin{prop}\label{prop:bigiprop} The complex ${\mathbf{I}}_{{}_X}^\bullet =F_*\Z_\W^\bullet[n]$ has the following properties:
\begin{enumerate}
\item ${\mathbf{I}}_{{}_X}^\bullet$ is the intersection cohomology complex with the constant $\Z$ local system.
\smallskip
\item The stalk cohomology of ${\mathbf{I}}_{{}_X}^\bullet$ is given by for all $x\in X$,
$$
H^k({\mathbf{I}}_{{}_X}^\bullet)_x\cong
\begin{cases}
\Z^{m(x)}, &\textnormal{ if } k=-n;\\
0, &\textnormal{ otherwise}.
\end{cases}
$$
\smallskip
\item The complex ${\mathbf{I}}_{{}_X}^\bullet$ is self-Verdier dual, i.e., 
$$\vdual {\mathbf{I}}_{{}_X}^\bullet \cong {\mathbf{I}}_{{}_X}^\bullet.$$
\smallskip
\item Suppose $x\in X$, and $j_x$ denotes the inclusion of $x$ into $X$. Then,
$$
j_x^!{\mathbf{I}}_{{}_X}^\bullet\cong \vdual j_x^*\vdual{\mathbf{I}}_{{}_X}^\bullet\cong \vdual j_x^*{\mathbf{I}}_{{}_X}^\bullet
$$
and so the costalk cohomology is given by
$$
H^k(j^!_x{\mathbf{I}}_{{}_X}^\bullet)\cong
\begin{cases}
\Z^{m(x)}, &\textnormal{ if } k=n;\\
0, &\textnormal{ otherwise}.
\end{cases}
$$
\smallskip
\item There is a canonical surjection of perverse sheaves $\Z_X[n]\xrightarrow{c} {\mathbf{I}}_{{}_X}^\bullet$, which induces the diagonal map on the stalk cohomology. 
\end{enumerate}
\end{prop}
\begin{proof} 

\ 

\begin{enumerate}

\item As $\Z_\W^\bullet[n]$ is the intersection cohomology complex on $\W$, with constant coefficients, and ${\mathbf{I}}_{{}_X}^\bullet$ is its push-forward by a finite map,  the support and cosupport conditions trivially push forward, and  ${\mathbf{I}}_{{}_X}^\bullet$ is an intersection cohomology complex on $X$. {\it A priori}, ${\mathbf{I}}_{{}_X}^\bullet$ could have ``twisted'' coefficients in a non-trivial local system on the regular part, $X_{\operatorname{reg}}$, of $X$. However, as we are assuming that $F$ is generically one-to-one, $F$ induces a homeomorphism when restricted to a map from a generic subset of $\W$ to a generic subset of $X_{\operatorname{reg}}$. Thus, on a generic subset of $X_{\operatorname{reg}}$, the complex ${\mathbf{I}}_{{}_X}^\bullet$ restricts to the shifted constant sheaf, and so ${\mathbf{I}}_{{}_X}^\bullet$ is the intersection cohomology complex with the constant local system.

Alternatively, $F$ is a small resolution of $X$, and so the push-forward of the shifted constant sheaf yields intersection cohomology. \cite[See][pg. 121]{inthom2}

\medskip

\item The formula for the stalk cohomology of ${\mathbf{I}}_{{}_X}^\bullet$ is immediate since ${\mathbf{I}}_{{}_X}^\bullet:=F_*\Z_\W^\bullet[n]$.

\medskip

\item With a field for the base ring, the self-duality of ${\mathbf{I}}_{{}_X}^\bullet$ would follow from its  being the intersection cohomology complex.  However, since we are using $\Z$ as our base ring, we use that
$$
\vdual {\mathbf{I}}_{{}_X}^\bullet\cong \vdual F_*\Z_\W^\bullet[n]\cong F_!\vdual\left(\Z_\W^\bullet[n]\right)\cong F_!\Z_\W^\bullet[n]\cong F_*\Z_\W^\bullet[n],
$$
where the last isomorphism follows from the fact that $F$ is finite, and hence proper.

\medskip

\item Using the self-duality of ${\mathbf{I}}_{{}_X}^\bullet$, we find
$$
j_x^!{\mathbf{I}}_{{}_X}^\bullet\cong \vdual j_x^*\vdual{\mathbf{I}}_{{}_X}^\bullet\cong \vdual j_x^*{\mathbf{I}}_{{}_X}^\bullet.
$$

Therefore,
$$
H^k(j_x^!{\mathbf{I}}_{{}_X}^\bullet) \cong H^k(\vdual j_x^*{\mathbf{I}}_{{}_X}^\bullet) \cong\operatorname{Hom}(H^{-k}(j_x^*{\mathbf{I}}_{{}_X}^\bullet), \Z)\oplus \operatorname{Ext}(H^{-k+1}(j_x^*{\mathbf{I}}_{{}_X}^\bullet) , \Z).
$$
Hence, using our earlier description of the stalk cohomology, we find
$$
H^k(j^!_x{\mathbf{I}}_{{}_X}^\bullet)\cong
\begin{cases}
\Z^{m(x)}, &\textnormal{ if } k=n;\\
0, &\textnormal{ otherwise}.
\end{cases}
$$

\item There is always a canonical morphism of perverse sheaves from the shifted constant sheaf to intersection cohomology with the (shifted) constant local system, i.e., a canonical morphism $\Z_X[n]\xrightarrow{c} {\mathbf{I}}_{{}_X}^\bullet$. 

Because we are using $\Z$ as our base ring, instead of a field, ${\mathbf{I}}_{{}_X}^\bullet$ is {\bf not} a simple object in the Abelian category of perverse sheaves of $\Z$-modules. However, ${\mathbf{I}}_{{}_X}^\bullet$ is nonetheless the intermediate extension of the constant sheaf on $X_{\operatorname{reg}}$, and so has no non-trivial sub-perverse sheaves or quotient perverse sheaves with support contained in $\Sigma X$. Therefore, since our morphism induces an isomorphism when restricted to $X_{\operatorname{reg}}$, its cokernel must be zero, i.e., the morphism $c$ is a surjection. 

The description of the induced map on the stalks follows at once from the fact that 
$${\mathbf{I}}_{{}_X}^\bullet=F_*\Z_\W^\bullet[n]\cong F_*F^*\Z_X^\bullet[n].$$

\end{enumerate}

\end{proof}

\bigskip

As an immediate corollary to Item (1) above, we have the well-known:

\begin{cor} There is a containment $D\subseteq \Sigma X$.
\end{cor}

\smallskip

The containment above can easily be strict; this is, for instance, the case when one parameterizes the cusp.

\bigskip

\begin{rem} We wish to make the costalk cohomology of a complex of sheaves more intuitive for the reader. We continue with the notation $j_x$ from the proposition. Let $B^\circ_\epsilon(x)$ denote an open ball of radius $\epsilon$, centered at $x\in X$. Suppose that $\Adot$ is a bounded constructible complex of sheaves on $X$.

Then, the cohomology of $j_x^!\Adot$ is isomorphic to the hypercohomology of a pair:
$$
H^k\left(j_x^!\Adot\right)\cong \hyp^k\left(B^\circ_\epsilon(x)\cap X, \big(B^\circ_\epsilon(x)-\{x\}\big)\cap X;\, \Adot\right),
$$
for $0<\epsilon\ll 1$, and there exists the usual long exact sequence for this pair, in which 
$$
\hyp^k\left(\big(B^\circ_\epsilon(x)-\{x\}\big)\cap X;\, \Adot\right)\cong \hyp^k\left(K_{X, x};\, \Adot\right).
$$

In particular,
$$
H^k\left(j_x^!\Z_X^\bullet[n]\right)\cong \widetilde H^{n+k-1}(K_{X, x}; \Z).
$$

\smallskip

Note that, as $\Z_X^\bullet[n]$ is a perverse sheaf, $H^k\left(j_x^!\Z_X^\bullet[n]\right)=0$ for $k\leq -1$. This is the cohomological manifestation of the fact that the real link of $X$ is $(n-2)$-connected \cite{milnorsing}.
\end{rem}

\bigskip

\section{The Multiple-Point Complex}

We let $\Ndot$ denote the kernel of the morphism $c$ from Property 5 in \propref{prop:bigiprop}, so that, in the Abelian category of perverse sheaves, we have a short exact sequence (which corresponds to a distinguished triangle in the derived category):
\begin{equation}\label{equ:ses}
0\rightarrow\Ndot\rightarrow \Z_X[n]\xrightarrow{c}{\mathbf{I}}_{{}_X}^\bullet\rightarrow 0.\tag{$\dagger$}
\end{equation}

\medskip

In our current setting, the morphism $c$ is particularly simple to describe on the level of stalk cohomology. Since  
$${\mathbf{I}}_{{}_X}^\bullet=F_*\Z_\W^\bullet[n]\cong F_*F^*\Z_X^\bullet[n],
$$
the morphism $c$ agrees with the natural map
$$
\Z_X[n]\xrightarrow{c}F_*F^*\Z_X^\bullet[n].
$$
On the stalk cohomology at $x$, this is just the diagonal inclusion $\Z\hookrightarrow \Z^{m(x)}$ in the only non-zero degree, $-n$. 
From the long exact sequence on stalk cohomology for our short exact sequence, we conclude that the perverse sheaf $\Ndot$  has stalk cohomology given by
$$
H^k(\Ndot)_x\cong
\begin{cases}
\Z^{m(x)-1}, &\textnormal{ if } k=-n+1;\\
0, &\textnormal{ otherwise}.
\end{cases}
$$
In particular, the support of $\Ndot$ is $D$. 
Note that the stalk cohomology of $\Ndot$ at $\0$ is $\Z^{r-1}$.

\medskip 

\begin{rem} The reader may be wondering why the morphism $c$ has a non-zero kernel in the category of perverse sheaves. After all, on the level of stalks, the map $c$ induces inclusions; it seems as though $c$ should have a non-trivial cokernel, not kernel.

It is true that there is a complex of sheaves $\Cdot$ and a distinguished triangle in the derived category
$$
\Z_X[n]\xrightarrow{c}{\mathbf{I}}_{{}_X}^\bullet\rightarrow \Cdot \xrightarrow{[1]}\Z_X[n]
$$
in which the stalk cohomology of $\Cdot$ is non-zero only in degree $-n$ and, in degree $-n$, is isomorphic to the cokernel of map induced on the stalks by $c$. However, the complex $\Cdot$ is {\bf not} perverse; it is supported on a set of dimension $n-1$ and has non-zero cohomology in degree $-n$.

However, we can ``turn'' the triangle to obtain a distinguished triangle
$$
\Cdot[-1]\rightarrow\Z_X[n]\xrightarrow{c}{\mathbf{I}}_{{}_X}^\bullet\xrightarrow{[1]} \Cdot,
$$
where $\Cdot[-1]$ is, in fact, perverse. Thus, in the Abelian category of perverse sheaves $\Ndot=\Cdot[-1]$ is the kernel of the morphism $c$.
\end{rem}

\smallskip

\noindent\rule{1in}{1pt}

\bigskip

\begin{defn}\label{defn:rmpc} We refer to the perverse sheaf $\Ndot$ as the {\bf multiple-point complex} (of $F$ on $X$).
\end{defn}

\medskip

We want to list the important features of the multiple-point complex which we have already discussed.

\smallskip

\begin{thm}\label{thm:kprop} The multiple-point complex $\Ndot$ has the following properties:

\begin{enumerate}
\item There is a short exact sequence in the Abelian category of perverse sheaves on $X$:
\begin{equation}\label{equ:ses2}
0\rightarrow\Ndot\rightarrow \Z^\bullet_X[n]\xrightarrow{c}F_*\Z_\W^\bullet[n]\rightarrow 0.\tag{$\ddagger$}
\end{equation}
In particular, $\Ndot$ is a perverse sheaf.
\smallskip
\item The support of $\Ndot$  is $D$.
\smallskip
\item The stalk cohomology of $\Ndot$ at a point $x\in D$ is zero in all degrees other than $-n+1$, and
$$
H^{-n+1}(\Ndot)_x\cong \Z^{m(x)-1}.
$$
In particular, the stalk cohomology of $\Ndot$ at the origin is $\Z^{r-1}$.
\smallskip
\item The costalk cohomology of $\Ndot$ is given by, for all $x\in X$,
$$
H^k(j_x^!\Ndot)\cong
\begin{cases}
\widetilde H^{n+k-1}(K_{X,x}; \Z), &\textnormal{ if } 0\leq k\leq n-1;\\
0, &\textnormal{ otherwise}.
\end{cases}
$$

\end{enumerate}

\end{thm}

\medskip

\begin{cor} $D$ is purely $(n-1)$-dimensional (which includes the possibility of being empty).
\end{cor}
\begin{proof} This is immediate from $D$ being the support of a perverse sheaf which, on an open dense set of $D$, has non-zero stalk cohomology precisely in degree $-n+1$.
\end{proof}
\bigskip

\medskip

We defer the proof of a technical lemma, referred to in the following definition, until \secref{sec:reduction}.
 
\medskip

\begin{defn}\label{def:reduced}
We define the {\bf $(r-1)$-reduced hypercohomology} $\widetilde\hyp^{k}(M_{h,\0}\cap D; \Ndot[-n+1])$ to be $H^{k}(\phi_h\Ndot[-n+1])_\0$ and note that this is justified by \lemref{lem:reduced}, since, with this definition,
\begin{itemize}
\item If $k\neq -1\textnormal{ or }0$, then 
$$\widetilde\hyp^{k}(M_{h,\0}\cap D; \Ndot[-n+1])\cong \hyp^{k}(M_{h,\0}\cap D; \Ndot[-n+1]).$$
\smallskip
\item There is an equality of Euler characteristics 
$$\chi\left(\widetilde\hyp^{*}(M_{h,\0}\cap D; \Ndot[-n+1])\right)=\chi\left(\hyp^{k}(M_{h,\0}\cap D; \Ndot[-n+1])\right)-(r-1).$$
\smallskip
\item  If $\dim_\0 D\cap V(h)\leq n-2$, then $\widetilde\hyp^{-1}(M_{h,\0}\cap D; \Ndot[-n+1])=0$ and 
$$\operatorname{rank}\widetilde\hyp^{0}(M_{h,\0}\cap D; \Ndot[-n+1])= \operatorname{rank}\hyp^{0}(M_{h,\0}\cap D; \Ndot[-n+1]) -(r-1).$$
\smallskip

\item If $r=1$, then $\widetilde\hyp^{-1}(M_{h,\0}\cap D; \Ndot[-n+1])=0$ and 
$$\widetilde\hyp^{0}(M_{h,\0}\cap D; \Ndot[-n+1])\cong \hyp^{0}(M_{h,\0}\cap D; \Ndot[-n+1]).$$
\end{itemize}

\end{defn}

\bigskip

The following theorem is now easy to prove.

\begin{thm}\label{thm:les} There is a long exact sequence, relating the Milnor fiber of $h$, the Milnor fibers of $h\circ F$, and the hypercohomology of the Milnor fiber of $h$ restricted to $D$ with coefficients in $\Ndot$, given by

\smallskip

$\cdots\rightarrow \widetilde\hyp^{j-1}(D\cap M_{h,\0}; \Ndot[-n+1])\rightarrow \widetilde H^{j}(M_{h, \0};\Z)\rightarrow\hfill$

\smallskip

$\hfill\bigoplus_i\widetilde H^{j}(M_{h\circ F, p_i};\Z)\rightarrow \widetilde\hyp^{j}(D\cap M_{h,\0}; \Ndot[-n+1])\rightarrow\cdots \ .$

\smallskip

This long exact sequence is compatible with the Milnor monodromy automorphisms in each degree.

\end{thm}
\begin{proof} We apply the exact functor $\phi_h[-1]$ to the short exact sequence \eqref{equ:ses2} which defines $\Ndot$ to obtain the following short exact sequence of perverse sheaves:
$$
0\rightarrow\phi_h[-1]\Ndot\rightarrow \phi_h[-1]\Z_X[n]\xrightarrow{\hat{c}}\phi_h[-1]F_*\Z_\W^\bullet[n]\rightarrow 0,
$$
where $\hat{c} = \phi_h[-1]c$. As the Milnor monodromy automorphism is natural, the maps in this short exact sequence commute with the Milnor monodromies.

By the induced long exact sequence on stalk cohomology and the lemma, we are finished.
\end{proof}

\medskip

In fact, the theorem gives us a refinement as to why one should think of $H^{k}(\phi_h\Ndot[-n+1])_\0$ as the $(r-1)$-reduced hypercohomology of $\hyp^{k}(M_{h,\0}\cap D; \Ndot[-n+1])$.

\begin{cor} Suppose that $\dim_\0\Sigma_{{}_{\operatorname{top}}}h\leq n-2$. Then, $\widetilde\hyp^{-1}(D\cap M_{h,\0}; \Ndot[-n+1])=0$ and $\widetilde\hyp^{0}(D\cap M_{h,\0}; \Ndot[-n+1])$ is free Abelian; consequently, $\widetilde\hyp^{0}(D\cap M_{h,\0}; \Ndot[-n+1])$ is obtained from $\hyp^{0}(D\cap M_{h,\0}; \Ndot[-n+1])$ by removing $r-1$ direct summands of $\Z$.

\end{cor}
\begin{proof} Note that $\dim_\0\Sigma_{{}_{\operatorname{top}}}h\leq n-2$ implies that $\dim_\0 V(h)<n$. Now, since $\dim_\0 V(h)<n$, $\bigoplus_i\widetilde H^{-1}(M_{h\circ F, p_i};\Z)=0$, and part of the long exact sequence from the theorem is

\medskip

$0\rightarrow  \widetilde\hyp^{-1}(D\cap M_{h,\0}; \Ndot[-n+1])\rightarrow\widetilde H^{0}(M_{h, \0};\Z)\rightarrow\hfill$

\medskip

$\hfill\bigoplus_i\widetilde H^{0}(M_{h\circ F, p_i};\Z)\rightarrow \widetilde\hyp^{0}(D\cap M_{h,\0}; \Ndot[-n+1])\rightarrow H^{1}(M_{h, \0};\Z)\rightarrow\cdots .$

\medskip

Each $\widetilde H^{0}(M_{h\circ F, p_i};\Z)$ is free Abelian and the Universal Coefficient Theorem for cohomology  tells us that $H^{1}(M_{h, \0};\Z)$ is free Abelian.

Since $\dim_\0\Sigma_{{}_{\operatorname{top}}}h\leq n-2$, \propref{prop:leprop} tells us that $\widetilde H^{0}(M_{h, \0};\Z)=0$, and we immediately conclude that $$ \widetilde\hyp^{-1}(D\cap M_{h,\0}; \Ndot[-n+1])=0$$
 and $\widetilde\hyp^{0}(D\cap M_{h,\0}; \Ndot[-n+1])$ is free Abelian. The final conclusion follows now from the splitting of the exact sequence in Item 3 of \lemref{lem:reduced}.
\end{proof}

\bigskip

\begin{cor}\label{cor:corles} If $S\cap\Sigma(h\circ F)=\emptyset$, then there is an isomorphism 
$$\widetilde H^{j}(M_{h, \0};\Z)\cong\widetilde\hyp^{j-1}(D\cap M_{h,\0}; \Ndot[-n+1])$$ 
and this isomorphism commutes with the Milnor monodromies.
\end{cor}

\bigskip

\begin{exm}\label{ex:unfold} Suppose that we have a finite map $f:(\V, S)\rightarrow (\Omega, \0)$, where $\V$ and $\Omega$ are open neighborhoods of $S$ in $\C^d$ and of the origin in $\C^{d+1}$, respectively. Suppose that $\T$ is an open neighborhood of the origin in $\C^d$, and that $F:\T\times\V\rightarrow \T\times\Omega$ is an unfolding of $f=f_{\0}$, i.e., $F$ is a finite analytic map of the form $F(\mbf t, \mbf v) \ = \ (\mbf t, f_{\mbf t}(\mbf v))$, where, for each $\mbf t\in\T$, $f_{\mbf t}$ is a finite map from $\V$ to $\Omega$.

Let $X$ denote the image of $F$, continue to write $F$ for the map from $\T\times\V$ to $X$, and let $h$ be the projection onto the first coordinate; thus, $(h\circ F)(t_1,\dots, t_d, \mbf v)=t_1$. Then, $S \cap \Sigma(h\circ F) = \emptyset$ and so $\widetilde H^{j}(M_{h, \0};\Z)$ is isomorphic to $\widetilde \hyp^{j-1}(D\cap M_{h,\0}; \Ndot[-n+1])$ by an isomorphism which commutes with the Milnor monodromies.

\end{exm}

\bigskip

Before we can prove the next corollary, we need to recall a lemma, which is well-known to experts in the field. See, for instance, \cite{dimcasheaves}, Theorem 4.1.22 (note that the setting of  \cite{dimcasheaves}, Theorem 4.1.22, is algebraic, but that assumption is used in the proof only to guarantee that there are a finite number of strata).

\begin{lem}\label{lem:addmulteuler}Let $\strat$ be a complex analytic Whitney stratification, with connected strata, of a complex analytic space $Y$. Suppose that $\strat$ contains a finite number of strata. Let $\Adot$ be a bounded complex of $\Z$-modules which is constructible with respect to $\strat$. For each stratum $S$, let $p_S$ denote a point in $S$.

Then, there is the following additivity/multiplicativity formula for the Euler characteristics:
$$
\chi\left(\hyp^*(Y; \Adot)\right)=\sum_S\chi(S)\chi(\Adot)_{p_S}.
$$
\end{lem}

\bigskip

\begin{cor}\label{cor:eulermulti} The relationship between the reduced Euler characteristics of the Milnor fiber of $h$ at $\0$, the Milnor fibers of $h\circ F$, and the $X_k$'s  is given by
$$
\widetilde\chi(M_{h,\0})= r-1 + \sum_i \widetilde\chi(M_{h\circ F,p_i}) - \sum_{k \geq 2} (k-1) \chi (X_k \cap M_{h,\0}) .
$$
\end{cor}

\begin{proof}
Via additivity of the Euler characteristic in the hypercohomology long exact sequence given in \thmref{thm:les}, we obtain the following relation:
\begin{align*}
\widetilde \chi (M_{h,\0})  &= \sum_i \widetilde \chi \left (M_{h \circ F,p_i} \right ) -\chi \left (\widetilde \hyp^{*}(D \cap M_{h,\0}; \Ndot[-n+1]) \right ) \\
&= r-1 +  \sum_i \widetilde \chi \left (M_{h \circ F,p_i} \right ) - \chi \left ( \hyp^{*}(D \cap M_{h,\0}; \Ndot[-n+1]) \right ).
\end{align*}
 We are then finished, provided that we show that 
$$\chi(D\cap M_{h,\0}; \Ndot[-n+1]) \ = \ \sum_{k\geq 2} (k-1)\chi(X_k\cap M_{h,\0}).$$

For this, we use \lemref{lem:addmulteuler}.  Take a complex analytic Whitney stratification $\strat'$ of $D$ such that $\Ndot_{|_D}$ is constructible with respect to $\strat'$; hence, for each $k$, $D\cap X_k$ is a union of strata. As $M_{h,\0}$ transversely intersects these strata, there is an induced Whitney stratification $\strat=\{S\}$ on $D\cap M_{h,\0}$ and also on each $D\cap X_k\cap M_{h,\0}$; these stratifications have a finite number of strata, since the Milnor fiber is defined inside a small ball and $\strat'$ is locally finite. 

Now, since the Euler characteristic of the stalk cohomology of $\Ndot[-n+1]$ at a point $x\in X_k$ is $(k-1)$, \lemref{lem:addmulteuler} yields
$$
\chi(D\cap M_{h,\0}; \Ndot[-n+1])=\sum_k\sum_{S\subseteq D\cap X_k\cap M_{h,\0}}(k-1)\chi(S).
$$
 Finally, , we ``put back together'' the Euler characteristics of the $X_k$'s, i.e., 
$$
\chi(X_k\cap M_{h,\0})=\sum_{S\subseteq D\cap X_k\cap M_{h,\0}}\chi(S),
$$
by again applying \lemref{lem:addmulteuler} to the constant sheaf on $X_k\cap M_{h,\0}$.
\end{proof}

\bigskip

\section{The Isolated Critical Point Case}

The case where $\0$ is an isolated point in $\Sigma_{{}_{\operatorname{top}}}h$ is of particular interest.

\medskip

\begin{thm}\label{thm:isol}
Suppose that $\0$ is an isolated point in $\Sigma_{{}_{\operatorname{top}}}h$. Then, for all $p_i \in S$, $\dim_{p_i}\Sigma(h\circ F) \leq 0$,  $\widetilde \hyp^{*}(D\cap M_{h,\0}; \Ndot[-n+1])$ is non-zero in (at most) one degree, degree $n-2$, where it is free Abelian, and the reduced, integral cohomology of $M_{h, \0}$ is non-zero in, at most, one degree, degree $n-1$, where it is free Abelian of rank
\begin{align*}
\mu_\0(h)  &=  \rank \widetilde \hyp^{n-2}(D\cap M_{h,\0}; \Ndot[-n+1])+\sum_i \mu_{p_i}(h\circ F)  \\
&= (-1)^{n-1}\Big[(r-1) -\sum_{k\geq 2} (k-1)\chi(X_k\cap M_{h,\0})\Big]+ \sum_i \mu_{p_i}(h\circ F).
\end{align*}

\medskip

In particular, if $\0$ is an isolated point in $\Sigma_{{}_{\operatorname{top}}}h$ and $S \cap \Sigma(h\circ F) = \emptyset$, then
$$
\mu_\0(h) = \rank \widetilde \hyp^{n-2}(D\cap M_{h,\0}; \Ndot[-n+1])= (-1)^{n-1}\Big[(r-1) -\sum_{k\geq 2} (k-1)\chi(X_k\cap M_{h,\0})\Big].
$$
\end{thm}
\begin{proof} Except for the last equalities in each line, this follows from \propref{prop:leprop} and $(\ast)$ in the proof of \thmref{thm:les}, since the hypothesis is equivalent to $\0$ being an isolated point in the support of $\phi_h[-1]\Z_X[n]$, and perverse sheaves which are supported at just an isolated point have non-zero stalk cohomology in only one degree, namely degree $0$.

The final equalities in each line follow from \corref{cor:eulermulti}.
\end{proof}

\bigskip

\begin{exm} Let us return to the situation of the unfolding situation in \exref{ex:unfold}, but now suppose that $F$ is a stable unfolding of $f$ with an isolated instability. Then, as before, letting $h$ be a projection onto an unfolding coordinate, $\0$ is an isolated point in $\Sigma_{{}_{\operatorname{top}}}h$ and $S \cap \Sigma(h\circ F) = \emptyset$.

Thus, the stable fiber has the cohomology of a finite bouquet of $(n-1)$-spheres, where the number of spheres, the Milnor number, is given by 
$$\rank \widetilde \hyp^{n-2}(D\cap M_{h,\0}; \Ndot[-n+1])= (-1)^{n-1}\Big[(r-1) -\sum_{k\geq 2} (k-1)\chi(X_k\cap M_{h,\0})\Big].
$$
Note, in particular, that this implies that the right-hand side is non-negative, which is distinctly non-obvious.

\medskip
Consider the simple, but illustrative, specific example where $r = 1$, $f(u)=(u^2, u^3)$, and the stable unfolding is given by $F(t,u)=(t, u^2-t, u(u^2-t))$. Let $X$ be the image of $F$, and let $h:X\rightarrow\C$ be the projection onto the first coordinate, so that $(h\circ F)(t,u)=t$. Note that, using $(t,x,y)$ as coordinates on $\C^3$, we have $X=V(y^2-x^3-tx^2)$.

Clearly $\0\not\in\Sigma(h\circ F)$, and $\0$ is an isolated point in $\Sigma_{{}_{\operatorname{top}}}h$. For $k\geq 2$, the only $X_k$ which is not empty is $X_2$, which equals the $t$-axis minus the origin. Furthermore, $X_2\cap M_{h, \0}$ is a single point.

We conclude from \thmref{thm:isol} that $M_{h, \0}$, which is the complex link of $X$, has the cohomology of a single $1$-sphere.
\end{exm} 

\medskip

As a further application, we recover a classical formula for the Milnor number, as given in Theorem 10.5 of \cite{milnorsing}:

\begin{thm}\label{thm:milnormulti}
Suppose that $n=2$ and that $F$ is a one-parameter unfolding of a parameterization $f$ of a plane curve singularity with $r$ irreducible components at the origin. Let $t$ be the unfolding parameter and suppose that the only singularities of $M_{t_{|_X}, \0}$ are nodes, and that there are $\delta$ of them.  Recall that $X=V(g)$, and let $g_0:=g_{|_{V(t)}}$. Then, the Milnor number of $g_0$ is given by the formula:
$$
\mu_\0 \left ( g_0 \right )= 2 \delta -r + 1.
$$
\end{thm}
\begin{proof}
We recall the following formula for the Milnor number of $g_{|_{V(t)}}$ at $\0$ \cite{lecycles}:
$$
\mu_\0 \left ( g_{|_{V(t)}} \right ) = \left ( \Gamma_{g,t}^1 \cdot V(t) \right )_\0 + \left ( \Lambda_{g,t}^1 \cdot V(t) \right )_\0,
$$
where $\Gamma_{g,t}^1$ is the relative polar curve of $g$ with respect to $t$, and $\Lambda_{g,t}^1$ is the one-dimensional L\^{e} cycle of $g$ with respect to $t$.  By our genericity assumption on the unfolding parameter $t$, we have $\left ( \Lambda_{g,t}^1 \cdot V(t) \right )_\0 = \delta$. Since the homotopy type of the complex link of $X$ at $\0$ is that of a bouquet of $\left ( \Gamma_{g,t}^1 \cdot V(t) \right )_\0$  $n$-spheres (see, for example, \cite{gencalcvan}), and we know that the unfolding function $F$ has an isolated instability at $\0$, it follows that this number of $n$-spheres is also given by the Milnor number $\mu_\0( t_{|_{X}} )$.  Consequently, \corref{cor:eulermulti} becomes
$$
\mu_\0(t_{|_X}) = -r+ 1 + \sum_{k \geq 2} (k-1) \chi(X_k \cap M_{t_{|_X},\0}).
$$
By assumption, $\chi(X_2 \cap M_{t_{|_X},\0})$ is the only non-zero summand in the above equation, and it is immediately seen to be the number of double points of $X \cap V(t)$ appearing in a stable perturbation.  Thus,
$$
\mu_\0 \left ( g_{|_{V(t)}} \right ) = 2 \delta - r + 1
$$
as desired. 
\end{proof}

\smallskip

\section{Questions and Future Directions}

If $\0$ is an isolated point in $\Sigma_{{}_{\operatorname{top}}}h$, then \thmref{thm:isol} provides a nice way of calculating the only non-zero the only non-zero cohomology group of the Milnor fiber of $h$.

However,  even if $S\cap\Sigma(h\circ F)=\emptyset$, it is unclear how much effectively calculable data about the cohomology of $M_{h,\0}$ one can extract from \corref{cor:corles} if $\dim_\0\Sigma_{{}_{\operatorname{top}}}h>0$  and $n\geq 3$ (so $\dim_\0D\geq 2$). Yes, we would know that
$$\widetilde H^{j}(M_{h, \0};\Z)\cong\widetilde\hyp^{j-1}(D\cap M_{h,\0}; \Ndot[-n+1]),$$ 
but this hypercohomology on the right is highly non-trivial to calculate. There is a spectral sequence that one could hope to use, but that does not seem to yield manageable data.

\medskip

So the question is: if $\dim_\0\Sigma_{{}_{\operatorname{top}}}h>0$  and $n\geq 3$, how do we say anything useful about $\widetilde\hyp^{j-1}(D\cap M_{h,\0}; \Ndot[-n+1])$?

\vskip 0.3in

An interesting direction of research might be to eliminate the finite map $F$ altogether. In the setting of this paper, the fact that the stalk cohomology of ${\mathbf{I}}_{{}_X}^\bullet$ is given by, for all $x\in X$,
$$
H^k({\mathbf{I}}_{{}_X}^\bullet)_x\cong
\begin{cases}
\Z^{m(x)}, &\textnormal{ if } k=-n\\
0, &\textnormal{ otherwise}
\end{cases}
$$
makes it seem as though it might be worthwhile to define, in general, {\it virtually parameterizable hypersurfaces} (VPHs) as those hypersurfaces for which the intersection cohomology has such a form. One could then study deformations of a given VPH via a family of VPHs.

\section{The $(r-1)$-reduction Lemma}\label{sec:reduction}

In this section, we prove a lemma which we used to justify the terminology ``$(r-1)$-reduced cohomology'' in \defref{def:reduced}. Note that, although $\Ndot$ is perverse, we use the shifted complex $\Ndot[-n+1]$ throughout so that the non-zero stalk cohomology is in degree $0$, in order to make using $\Ndot[-n+1]$ easier to think of as just using constant $\Z$ coefficients, but with multiplicities.

We remind the reader of our earlier convention:  the reduced cohomology of the empty set, $\widetilde H^k(\emptyset; \Z)$, is zero in all degrees other than degree $-1$, and $\widetilde H^{-1}(\emptyset; \Z)=\Z$.

\begin{lem}\label{lem:reduced}

\ 

\begin{enumerate} 
\item For all $k$,
$$H^k(\phi_h\Z^\bullet_X)_\0\cong\widetilde H^{k}(M_{h, \0};\Z),$$
 which is possibly non-zero only for $n-s-1\leq k\leq n-1$, where $s:=\dim_\0\Sigma_{{}_{\operatorname{top}}}h\leq n$. Furthermore, when $k=-1$, this cohomology is non-zero  if and only if $h$ is identically zero (so that $M_{h, \0}=\emptyset$).

\medskip

\item For all $k$,
$$H^k(\phi_hF_*\Z_\W^\bullet)_\0\cong \bigoplus_i \widetilde H^{k}(M_{h\circ F, p_i};\Z),$$ which is possibly non-zero only for \hbox{$n-\hat s-1\leq k\leq n-1$}, where 
$$\hat s:=\operatorname{max}_i\{\dim_{p_i}\Sigma(h\circ F)\}\leq n.$$
 Furthermore, when $k=-1$, this cohomology is non-zero  if and only if $h$ is identically zero on at least one irreducible component of $X$.

\medskip

\item $H^{k}(\phi_h\Ndot[-n+1])_\0$ is possibly non-zero only for $-1\leq k\leq n-2$. Furthermore, if $h$ is not identically zero on any irreducible component of $D$, i.e., if $\dim_\0 D\cap V(h)\leq n-2$, then $H^{-1}(\phi_h\Ndot[-n+1])_\0=0$.

\smallskip

We also have:

\smallskip

\begin{itemize}
\item For $k\neq -1\textnormal{ or }0$,
$$H^{k}(\phi_h\Ndot[-n+1])_\0\cong \hyp^{k}(M_{h,\0}\cap D; \Ndot[-n+1]).$$

\smallskip

\item If $r=1$, then, for all $k$,
$$
H^{k}(\phi_h\Ndot[-n+1])_\0 \cong  \hyp^{k}(M_{h,\0}\cap D; \Ndot[-n+1]).
$$ 

\smallskip

\item There is an exact sequence

\smallskip

$0\rightarrow H^{-1}(\phi_h[-1]\Ndot[-n+1])_\0\rightarrow \Z^{r-1}\rightarrow\hfill$
\smallskip
$\hfill \hyp^{0}(M_{h,\0}\cap D; \Ndot[-n+1])\rightarrow H^{0}(\phi_h[-1]\Ndot[-n+1])_\0\rightarrow 0.$

\end{itemize}

\end{enumerate}
\end{lem}

\begin{proof}
Let $B$ denote a small open ball around the origin in $\C^{n+1}$. Then, for every bounded, constructible complex $\Adot$ on $X$, if we let $Y=\operatorname{supp}\Adot$, then
$$
H^k(\phi_h[-1]\Adot)_\0\cong\hyp^k(B\cap X, M_{h,\0}; \Adot)\cong  \hyp^k(B\cap Y, M_{h,\0}\cap Y ; \Adot),
$$
where this hypercohomology group fits into the hypercohomology long exact sequence of the pair $(B\cap X, M_{h,\0})$.

Consequently, using $\Adot = \Z^\bullet_X[n]$, we find that $H^k(\phi_h[-1]\Z^\bullet_X[n])_\0$ is, in fact, equal to the standard reduced cohomology of the Milnor fiber $\widetilde H^{k+n-1}(M_{h, \0};\Z)$, {\bf provided} that we use our convention on the reduced cohomology of the empty set.

\smallskip

Suppose instead that we use $\Adot=F_*\Z_\W^\bullet[n]$. Then, by the base change formula \cite{kashsch}, $\phi_h[-1]F_*\Z_\W^\bullet[n]$ is naturally isomorphic to $\widehat F_*\phi_{h\circ F}[-1]\Z_\W^\bullet[n]$, where $\widehat F$ denotes the map induced by $F$ from $F^{-1}h^{-1}(0)$ to $h^{-1}(0)$.

Therefore,
$$
H^k(\phi_h[-1]F_*\Z_\W^\bullet[n])_\0\cong \bigoplus_i H^k(\phi_{h\circ F}[-1]\Z_\W^\bullet[n])_{p_i},
$$
which, from our work above, implies that
$$
H^k(\phi_h[-1]F_*\Z_\W^\bullet[n])_\0\cong \bigoplus_i \widetilde H^{k+n-1}(M_{h\circ F, p_i};\Z).
$$

\smallskip

Now we need to look at the more complicated case where $\Adot=\Ndot$. We find
$$
H^k(\phi_h[-1]\Ndot)_\0\cong \hyp^k(B\cap D, M_{h,\0}\cap D ; \Ndot),
$$
and the long exact sequence of the pair, together with the fact that we know $\hyp^*(B\cap D;\Ndot)\cong H^*(\Ndot)_\0$, gives us the exact sequence
$$
\cdots\rightarrow \hyp^{-n}(M_{h,\0}\cap D; \Ndot)\rightarrow H^{-n+1}(\phi_h[-1]\Ndot)_\0\rightarrow \Z^{r-1}\rightarrow \hyp^{-n+1}(M_{h,\0}\cap D; \Ndot)\rightarrow 
$$
$$
H^{-n+2}(\phi_h[-1]\Ndot)_\0\rightarrow 0\rightarrow \hyp^{-n+2}(M_{h,\0}\cap D; \Ndot)\rightarrow H^{-n+3}(\phi_h[-1]\Ndot)_\0\rightarrow 0\rightarrow\cdots.
$$

\medskip

We claim that $\hyp^{k}(M_{h,\0}\cap D; \Ndot)=0$ for all $k\leq -n$. This follows immediately from the fact that 
$\hyp^{-n}(M_{h,\0}\cap D; \Ndot)\cong H^{-n+1}(\psi_h[-1]\Ndot)_\0$, and $\psi_h[-1]\Ndot$ is a perverse sheaf supported on $\overline{D-V(h)}\cap V(h)$, which has dimension less than or equal to $n-2$. Since we also know that $H^k(\Ndot)_\0=0$ for all $k\leq -n$, we conclude the following:

\begin{itemize}
\smallskip
\item For all $k\leq -n$, $H^k(\phi_h[-1]\Ndot)_\0\cong \hyp^{k}(M_{h,\0}\cap D; \Ndot)=0$.
\smallskip
\item For all $k\geq -n+3$, 
$$H^k(\phi_h[-1]\Ndot)_\0\cong \hyp^{k-1}(M_{h,\0}\cap D; \Ndot)\cong \hyp^{k+n-2}(M_{h,\0}\cap D; \Ndot[-n+1]).$$
\smallskip
\item We have an exact sequence
$$
0\rightarrow H^{-n+1}(\phi_h[-1]\Ndot)_\0\rightarrow \Z^{r-1}\rightarrow \hyp^{-n+1}(M_{h,\0}\cap D; \Ndot)\rightarrow H^{-n+2}(\phi_h[-1]\Ndot)_\0\rightarrow 0.
$$
\end{itemize}

If $\dim_\0 D\cap V(h)\leq n-2$, then $\phi_h[-1]\Ndot$ is a perverse sheaf which is supported on a set of dimension at most $n-2$; the stalk cohomology in degrees less than $-(n-2)=-n+2$ is zero.  Therefore, in this case, 
$$H^{-n+1}(\phi_h[-1]\Ndot)_\0\cong H^{-1}(\phi_h\Ndot[-n+1])_\0=0.$$
\end{proof}

\bigskip

\printbibliography
%\bibliographystyle{plain}
%\bibliographystyle{amsalpha}
%\bibliography{Masseybib}
%\printindex
\end{document}